\documentclass[11pt,leqno]{amsart}

\usepackage{amsthm}
\usepackage{amsmath}
\usepackage{amssymb}
\usepackage[all,cmtip]{xy}
\usepackage[T1]{fontenc}
\usepackage{mathrsfs}
\usepackage{enumerate}

\usepackage{color}

\theoremstyle{plain}
\newtheorem{thm}{Theorem}

\newtheorem{prop}{Proposition}
\newtheorem{cor}{Corollary}
\theoremstyle{definition}
\newtheorem{defn}{Definition}
\newtheorem{exmp}{Example}
\newtheorem{rem}{Remark}

\newcommand{\Z}{\mathbb{Z}}
\newcommand{\C}{\mathbb{C}}
\newcommand{\N}{\mathbb{N}}

\DeclareMathOperator{\ob}{ob}
\DeclareMathOperator{\mor}{mor}

\DeclareMathOperator{\Gal}{Gal}



\thanks{The authors wish to thank Jonas T. Hartwig for stimulating and useful discussions on the topic of this paper.
The first author was partially supported by The Swedish Research Council, The Swedish Foundation for International Cooperation in Research and Higher Education (STINT), The Crafoord Foundation, The Royal Physiographic Society in Lund, The Swedish Royal Academy of Sciences and "LieGrits", a Marie Curie Research Training Network funded by the European Community as project MRTN-CT 2003-505078.
}

\title{The Ideal Intersection Property for Groupoid Graded Rings}
\date{}

\author{Johan \"{O}inert and Patrik Lundstr\"{o}m}

\address{Johan \"{O}inert,
Department of Mathematical Sciences,
University of Copenhagen,
Universitetsparken 5,
DK-2100 Copenhagen \O,
Denmark}
\email{oinert@math.ku.dk}

\address{Patrik Lundstr\"{o}m,
University West, Department of Engineering Science,
SE-46186 Trollh\"{a}ttan, Sweden}
\email{Patrik.Lundstrom@hv.se}

\begin{document}

\maketitle

\begin{abstract}
We show that if a groupoid graded ring has
a grading satis\-fying a certain nondegeneracy property,
then the commutant of the center of the principal component of the ring
has the ideal intersection property, that is it
intersects nontrivially every nonzero ideal of the ring.
Furthermore, we show that for skew groupoid algebras with
commutative principal component, the principal component
is maximal commutative if and only if it has the
ideal intersection property.
\end{abstract}

\section{Introduction}

Let $R$ be a ring.
By this we always mean that $R$ is an additive group
equipped with a multiplication which is associative and unital.
If $X$ and $Y$ are nonempty subsets of $R$,
then $XY$ denotes the set of all finite sums of
elements of the form $xy$ with $x \in X$ and $y \in Y$.
The identity element of $R$ is denoted $1_R$ and
is always assumed to be nonzero.
We always assume that ring homomorphisms respect the multiplicative identities.
By the \emph{commutant}
of a subset $S$ of a ring $R$, denoted $C_R(S)$, we mean the set of elements
of $R$ that commute with each element of $S$;
the commutant $C_R(R)$ is called the \emph{center} of $R$ and will be denoted by $Z(R)$.

Suppose that $R'$ is a subring of $R$, i.e.
there is an injective ring homorphism  $R' \rightarrow R$.
By abuse of notation, we will, in that case,
often identify $R'$ with its image in $R$.
We say that $R'$ has the \emph{ideal intersection
property} in $R$ if $R' \cap I \neq \{ 0 \}$
for each nonzero ideal $I$ of $R$.
An important ring theoretic impetus for studying the ideal intersection
property is that it can be characterized
by saying that any ring homomorphism
$R \rightarrow S$ is injective whenever the
induced restriction $R' \rightarrow S$ is injective (see Section \ref{Application_Injectivity} for details).
Recall that if $R'$ is commutative, then it
is called a \emph{maximal commutative subring} of $R$
if $R'=C_R(R')$, i.e. $R'$ coincides with its commutant in $R$.
A lot of work has been devoted to the investigation of the
connection between maxi\-mal commutativity
and the ideal intersection property
(see \cite{coh}, \cite{fis}, \cite{for78}, \cite{lor},
\cite{lor79}, \cite{lor80}, \cite{mon78} and \cite{pas77}).
Recently (see \cite{oin06}, \cite{oin07}, \cite{oin08},
\cite{oin10} and \cite{oin09})
such a connection was established for the commutant
of the neutral component of strongly group graded rings
and crystalline graded rings
(see Theorem \ref{TheoremOne} and Theorem \ref{TheoremTwo} below).

Let $G$ be a group with neutral element $e$.
Recall that $R$ is said to be \emph{graded} by the group $G$
if there is a set of additive subgroups,
$R_s$, for $s \in G$, of $R$ such that $R =
\bigoplus_{s \in G} R_s$ and $R_s R_t \subseteq
R_{st}$, for $s,t \in G$.
The subring $R_e$ of $R$ is referred to as the \emph{neutral component} of $R$.
Let $R$ be a ring graded by the group $G$.
If $R_s R_t = R_{st}$, for $s,t \in G$,
then $R$ is called strongly graded.
For more details concerning group graded rings,
see e.g. \cite{nas}.

\begin{thm}\label{TheoremOne}
If a strongly group graded ring has a commutative
neutral component, then the commutant of
the neutral component has the ideal
intersection property.
\end{thm}

For a proof of Theorem \ref{TheoremOne},
see \cite[Theorem 4]{oin10}.
Recall from \cite{nauoyst} that a ring $R$ graded
by the group $G$ is called
crystalline graded if for each $s \in G$ there
is a nonzero element in $R_s$ with the
property that it freely generates $R_s$ both as
a left and a right $R_e$-module.

\begin{thm}\label{TheoremTwo}
If a crystalline graded ring has a commutative
neutral component, then the commutant of
the neutral component has the ideal
intersection property.
\end{thm}

For a proof of Theorem \ref{TheoremTwo},
see \cite[Theorem 5]{oin10}.
The main purpose of this article is to simultaneously
generalize Theorem \ref{TheoremOne} and Theorem \ref{TheoremTwo}
as well as extending the grading from
groups to groupoids in the following way.

\begin{thm}\label{maintheorem}
If a groupoid graded ring has
a right (or left) nondegenerate grading,
then the commutant of the center of
the principal component of the ring
has the ideal intersection property.
\end{thm}

For the definitions of \emph{principal component}, groupoid graded rings
and right (and left) \emph{nondegeneracy} of a grading,
see Definition \ref{defgradedring} and
Definition \ref{defnonzeroidealproperty}.
For a proof of Theorem \ref{maintheorem}, see Section \ref{ideals}.
In general, the ideal intersection property alone
is of course not a sufficient condition
for the principal component in a graded ring to be maximal commutative
(see e.g. \cite[Example 2.3]{oin07}).
The secondary purpose of this article
is to show that this is however true for a large class
of skew category algebras.

\begin{thm}\label{skewcatalg}
Suppose that $R$ is a skew category
algebra with a commutative principal
component $A$ and that the grading of $R$ is defined by a
category with finitely many objects.
If $A$ has the ideal intersection property in $R$,
then $A$ is maximal commutative in $R$.
\end{thm}

For the definition of skew category algebras,
see Example \ref{example3} in Section \ref{gradedrings}.
For a proof of Theorem \ref{skewcatalg},
see the end of Section \ref{ideals}.
By Theorem \ref{maintheorem} and Theorem \ref{skewcatalg}
the succeeding result follows immediately.

\begin{thm}\label{skewgroupoidtheorem}
Suppose that $R$ is a skew groupoid
algebra with a commutative principal
component $A$ and that the grading of $R$ is defined by a
groupoid with finitely many objects.
Then $A$ is maximal commutative in $R$
if and only if $A$ has the ideal intersection
property in $R$.
\end{thm}

Note that Theorem \ref{skewgroupoidtheorem} simultaneously generalizes
\cite[Theorem 3.4]{oinert09}, \cite[Corollary 6]{oinlun08}
and \cite[Proposition 10]{oinlun08}.

\section{Graded Rings}\label{gradedrings}

In this section, we recall the definition of
category graded rings from \cite{lu06}
(see Definition \ref{defgradedring}) and we give some examples
of such rings (see Example \ref{example1}--\ref{example5}).
In the end of this section, we show a result
(see Proposition \ref{identityprop}) concerning
identity elements in category graded rings
and category crossed products
for use in Section \ref{ideals}.

Suppose that $G$ is a category.
The family of objects of $G$ is denoted $\ob(G)$;
we will often identify an object in $G$ with
its associated identity morphism.
The family of morphisms in $G$ is denoted $\ob(G)$;
by abuse of notation, we will often write $s \in G$
when we mean $s \in \mor(G)$.
The domain and codomain of a morphism $s$ in $G$ is denoted
$d(s)$ and $c(s)$ respectively.
We let $G^{(2)}$ denote the collection of composable
pairs of morphisms in $G$, i.e. all $(s,t)$ in
$\mor(G) \times \mor(G)$ satisfying $d(s)=c(t)$.
A category is called cancellable (a groupoid)
if all its morphisms are both monomorphisms
and epimorphisms (isomorphisms).
A category is called a monoid if
it only has one object.

\begin{defn}\label{defgradedring}
Let $R$ be a ring and $G$ a category.
A set of additive subgroups,
$R_s$, for $s \in G$, of $R$ is said to be a \emph{$G$-filter} in $R$ if for all $s ,t \in G$,
we have $R_s R_t \subseteq R_{st}$ if
$(s,t) \in G^{(2)}$ and
$R_s R_t = \{ 0 \}$ otherwise.
The ring $R$ is said to be \emph{graded} by the category $G$
if there is a $G$-filter,
$R_s$, for $s \in G$, in $R$ such that
$R = \bigoplus_{s \in G} R_s$.
If $R$ is graded by $G$ and
$R_s R_t = R_{st}$, for $(s,t) \in G^{(2)}$,
then $R$ is called \emph{strongly graded}.
By the \emph{principal component} of $R$ we mean
the set $R_0 := \bigoplus_{e \in \ob(G)} R_e$.
If $X$ is a subset of $R$ we put $X_s = X \cap R_s$,
for $s \in G$; the set $X$ is said to be \emph{homogeneous}
if $X = \sum_{s \in G} X_s$.
A subring of $R$ is said to be \emph{graded} if it
is homogeneous as a subset of $R$.
\end{defn}

\begin{rem}
We always assume that filters and gradings are
defined over small categories, that is categories
where the collection of morphisms is a set.
\end{rem}

\begin{rem}
Filtrations of rings, in the sense of e.g. \cite{Lang}, are examples of $G$-filters with
$G=(\N,+)$.
\end{rem}

\begin{rem}
Groupoid graded rings have recently arisen as natural objects
of study in Galois theory for weak Hopf algebras (see e.g. \cite{cae2004}).
\end{rem}

Category graded rings are very general mathematical objects.
To give some flavor of this we now show some examples of such rings.

\begin{exmp}\label{example1}
Recall from \cite{oinlun08} that category crossed products are defined by
first specifying a crossed system i.e. a quadruple
$\{ A,G,\sigma,\alpha \}$ where
$A$ is the direct sum of rings $A_e$, $e \in \ob(G)$,
$\sigma_s : A_{d(s)} \rightarrow A_{c(s)}$, for $s \in G$,
are ring homomorphisms and
$\alpha$ is a map from $G^{(2)}$ to the
disjoint union of the sets $A_e$, for $e \in \ob(G)$,
with $\alpha(s,t) \in A_{c(s)}$,
for $(s,t) \in G^{(2)}$, satisfying the following five conditions:
\begin{equation}\label{idd}
\sigma_e = {\rm id}_{A_e}
\end{equation}
\begin{equation}\label{identityr}
\alpha(s,d(s)) = 1_{A_{c(s)}}
\end{equation}
\begin{equation}\label{identityl}
\alpha(c(t),t) = 1_{A_{c(t)}}
\end{equation}
\begin{equation}\label{associativee}
\alpha(s,t) \alpha(st,r) = \sigma_s(\alpha(t,r)) \alpha(s,tr)
\end{equation}
\begin{equation}\label{algebraa}
\sigma_s(\sigma_t(a)) \alpha(s,t) = \alpha(s,t) \sigma_{st}(a)
\end{equation}
for all $e \in \ob(G)$, all $(s,t,r) \in G^{(3)}$ and all $a \in A_{d(t)}$.
Let $A \rtimes_{\alpha}^{\sigma} G$ denote the collection of formal sums
$\sum_{s \in G} a_s u_s$, where $a_s \in A_{c(s)}$, $s \in G$,
are chosen so that all but finitely many of them are zero.
Define addition on $A \rtimes_{\alpha}^{\sigma} G$
by
\begin{equation}\label{addition}
\sum_{s \in G} a_s u_s + \sum_{s \in G} b_s u_s =
\sum_{s \in G} (a_s + b_s)u_s
\end{equation}
and define multiplication on $A \rtimes_{\alpha}^{\sigma} G$ by
the bilinear extension of the relation
\begin{equation}\label{multiplication}
(a_s u_s)(b_t u_t) = a_s \sigma_s(b_t) \alpha(s,t) u_{st}
\end{equation}
if $(s,t) \in G^{(2)}$ and $(a_s u_s)(b_t u_t) = 0$ otherwise
where $a_s \in A_{c(s)}$ and $b_t \in A_{c(t)}$.
By (\ref{associativee}) the multiplication on
$A \rtimes_{\alpha}^{\sigma} G$ is associative.
We will often identify $A$ with
$\bigoplus_{e \in \ob(G)} A_e u_e$; this ring will be referred
to as the coefficient ring of $A \rtimes_{\alpha}^{\sigma} G$.
It is clear that $A \rtimes_{\alpha}^{\sigma} G$ is a category graded ring
and it is strongly graded by $G$
if and only if each $\alpha(s,t)$, for $(s,t) \in G^{(2)}$,
has a left inverse in $A_{c(s)}$.
\end{exmp}

\begin{exmp}\label{example2}
We say that $A \rtimes_{\alpha}^{\sigma} G$ is a
twisted category algebra
if each $\sigma_s$, $s \in G$, with $d(s)=c(s)$
equals the identity map on $A_{d(s)}=A_{c(s)}$;
in that case the category crossed product is
denoted $A \rtimes_{\alpha} G$.
Now we consider two well known special
cases of this construction when
all the rings $A_e$, for $e \in \ob(G)$,
coincide with a fixed ring $D$.

If $G$ is a group, then $D \rtimes_{\alpha} G$
coincides with the usual construction of a twisted group algebra.

If $n$ is a positive integer and we define
$G$ to be the groupoid with the $n$ first positive integers as objects and
as arrows all ordered pairs $(i,j)$, for $1 \leq i,j \leq n$,
equipped with the partial binary operation
defined by letting $(i,j)(k,l)$ be defined and equal to $(i,l)$
precisely when $j=k$, then
$D \rtimes_{\alpha} G$ is the ring of
twisted square matrices over $D$ of size $n$.
\end{exmp}

\begin{exmp}\label{example3}
We say that $A \rtimes_{\alpha}^{\sigma} G$
is a skew category algebra
if $\alpha(s,t) = 1_{A_{c(s)}}$ for all $(s,t) \in G^{(2)}$;
in that case the category crossed product is
denoted $A \rtimes^{\sigma} G$.
Now we consider two well known special
cases of this construction when
all the rings $A_e$, for $e \in \ob(G)$,
coincide with a fixed ring $D$.

If $G$ is a group, then $D \rtimes^{\sigma} G$
is the usual skew group ring of $G$ over $D$.

Even in the case when $\sigma_s = {\rm id}_D$
for all $s \in G$, we get interesting examples of algebras.
Indeed, if $G$ is a a quiver, i.e. a directed graph
where loops and multiple arrows between two vertices are allowed,
then the category crossed product
coincides with the quiver algebra $DG$.
If, on the other hand, $G$ equals the groupoid in the
end of Example \ref{example2}, then
the category crossed product coincides
with the ring of $n \times n$ matrices over $D$.
\end{exmp}

\begin{exmp}\label{example4}
Recall from \cite{lu05} the construction of
crossed product algebras $F \rtimes_{\alpha}^{\sigma} G$
defined by finite separable (not necessarily normal) field extensions $L/K$.
Let $N$ denote a normal closure of $L/K$ and
let $\Gal(N/K)$ denote the Galois group of $N/K$.
Furthermore, let $F$ denote the direct sum of the conjugate fields
$L_i$, for $i = 1, \ldots , n$; put $L_1 = L$.
If $1 \leq i,j \leq n$, then let $G_{ij}$ denote the set of
field isomorphisms from $L_j$ to $L_i$. If $s \in G_{ij}$,
then we indicate this by writing $d(s) = j$ and $c(s) = i$.
If we let $G$ be the union of the $G_{ij}$, for $1 \leq i,j \leq n$, then $G$ is a groupoid.
For each $s \in G$, let $\sigma_s = s$.
Suppose that $\alpha$ is a map
$G^{(2)} \rightarrow \bigsqcup_{i=1}^n L_i$
with $\alpha(s,t) \in L_{c(s)}$, for $(s,t) \in G^{(2)}$ satisfying
(\ref{identityr}), (\ref{identityl}) and (\ref{associativee})
for all $(s,t,r) \in G^{(3)}$ and all $a \in L_{d(t)}$.
If $L/K$ is normal, then $G = \Gal(N/K)$
and hence $F \rtimes_{\alpha}^{\sigma} G$
coincides with the usual crossed product
defined by the Galois extension $L/K$.
\end{exmp}

\begin{exmp}\label{example5}
Suppose that $K$ is a commutative ring and that $M_n(K)$
denotes the ring of square matrices of size $n$ over $K$.
Furthermore, let $e_{ij}$ denote
the matrix in $M_n(K)$ with 1 in the $ij$:th
position and 0 elsewhere.
By an example of Dade \cite{dade} the decomposition
$R := M_3(K) = R_0 \oplus R_1$, where
$$R_0 = Ke_{11} + Ke_{22} + Ke_{23} + Ke_{32} + Ke_{33}$$
and
$$R_1 =Ke_{21} + Ke_{31} + Ke_{12} + Ke_{13},$$
is an example of a strongly ${\Bbb Z}_2$-graded ring
which is not a group crossed product with this grading. In fact,
since ${\rm dim}_K(R_0) = 5 \neq 4 = {\rm dim}_K(R_1)$,
it follows that $R$ is not even crystalline graded.
Inspired by this example, we now show that there
are nontrivial strongly groupoid graded rings which are not,
in a natural way, groupoid crossed products
in the sense defined in Example \ref{example1}.
Indeed, let $G$ be the unique thin
connected groupoid with two objects.
Explicitly this means that the morphisms of $G$ are $e$, $f$,
$s$ and $t$; multiplication is defined by the relations
$e^2 = e$, $f^2 = f$, $es = s$, $te = t$,
$sf = s$, $ft = t$, $st = e$ and $ts = f$.
Define a 13-dimensional
$K$-subalgebra $R$ of $M_5(K)$ by
$R = R_e \oplus R_f \oplus R_s \oplus R_t$ where
$$
\begin{array}{ll}
R_e = Ke_{11} + Ke_{33} & R_f = Ke_{22} + Ke_{44} + Ke_{45} + Ke_{54} + Ke_{55} \\
R_s = Ke_{12} + Ke_{34} + Ke_{35} & R_t = Ke_{21} + Ke_{43} + Ke_{53}
\end{array}
$$
A straightforward calculation shows that
$$R_e R_e = R_e, \ R_f R_f = R_f, \ R_e R_s = R_s, \
R_t R_e = R_t, \ R_s R_f = R_s, \ R_f R_t = R_t.$$
Therefore $R$ is strongly graded by $G$.
However, since ${\rm dim}_K(R_f) = 5 \neq 3 = {\rm dim}_K(R_t)$, the
left $R_f$-module $R_t$ can not be free on one generator.
\end{exmp}

For use in the next section, we gather some facts
concerning identity elements in category graded rings.

\begin{prop}\label{identityprop}
Suppose that $R$ is a ring graded by a category $G$.
\begin{enumerate}[{\rm (a)}]
\item If $G$ is cancellable, then $1_R \in R_0$.
In particular, the same conclusion holds if $G$ is a groupoid.

\item If $R = A \rtimes_{\sigma}^{\alpha} G$ is
a category crossed product algebra, then $R$ has an identity
element if and only if $G$ has finitely many objects;
in that case $1_R = \sum_{e \in \ob(G)} u_e$.
\end{enumerate}
\end{prop}

\begin{proof}
(a) Let $1_R=\sum_{s\in G}1_s$ where $1_s \in R_s$ for $s\in G$. If $t \in G$, then
$1_t = 1_R 1_t = \sum_{s \in G} 1_s 1_t$.
Since $G$ is cancellable, this implies that
$1_s 1_t = 0$ whenever $s \in G \setminus \ob(G)$.
Therefore, if $s \in G \setminus \ob(G)$, then
$1_s = 1_s 1_R = \sum_{t \in G} 1_s 1_t = 0$.
The last part follows from the fact that groupoids
are cancellable categories.

(b) First we show the ''if'' statement.
If $G$ has finitely many objects, then,
by equations (\ref{identityr}) and (\ref{identityl})
from Example \ref{example1}, it
immediately follows that $\sum_{e \in \ob(G)} u_e$
is an identity element of $R$.
Now we show the ''only if'' statement.
Suppose that $1_R = \sum_{s \in G} a_s u_s$
for some $a_s \in A_{c(s)}$, $s \in G$,
satisfying $a_s = 0$ for all but finitely
many $s \in G$. Since $u_e = 1_R u_e = u_e 1_R$
holds for all $e \in \ob(G)$, it follows
that $a_s = 0$ whenever $c(s) \neq d(s)$.
Since $u_e = u_e 1_R u_e$ holds for all
$e \in \ob(G)$, it follows that
$a_s = 0$ whenever $s \notin \ob(G)$
and $a_e = 1_{A_e} \neq 0$ for all $e \in \ob(G)$.
Therefore $\ob(G)$ is finite
and $1_R = \sum_{e \in \ob(G)} u_e$.
\end{proof}

\section{Commutants and Ideals}\label{ideals}

In this section, we prove Theorem \ref{maintheorem} and
Theorem \ref{skewcatalg}.
To this end, we first gather some facts concerning
commutants in graded rings.

\begin{prop}\label{gradedcommutant}
Suppose that $R$ is a ring graded by a category $G$.
\begin{enumerate}[{\rm (a)}]
\item If $X$ is a subset of $R$, then
$\{ C_R(X)_s \mid s \in G \}$
is a $G$-filter in $R$.

\item If $A$ is a homogeneous additive subgroup of $R$,
then $$C_R(A)_s  = \bigcap_{u \in G} C_{R_s}(A_u)$$
for all $s \in G$.

\item The set $C_R(R_0)$ is a
graded subring of $R$ with
$$C_R(R_0)_s =
\left\{
\begin{array}{l}
C_{R_s}(R_{d(s)}), \ {\rm if} \ c(s)=d(s), \\
\left\{ r_s \in R_s \mid R_{c(s)}r_s = r_s R_{d(s)}
= \{ 0 \} \right\}, \ {\rm otherwise.} \\
\end{array}
\right.$$

\item If $1_R \in R_0$, then $C_R(R_0)$ is a
graded subring of $R$ with
$$C_R(R_0)_s =
\left\{
\begin{array}{l}
C_{R_s}(R_{d(s)}), \ {\rm if} \ c(s)=d(s), \\
\{ 0 \}, \ {\rm otherwise.} \\
\end{array}
\right.$$ In particular, if $G$ is cancellable
or $R$ is a category crossed product and $G$ has finitely many objects, then the same
conclusion holds.
\end{enumerate}
\end{prop}

\begin{proof}
(a) Take $s,t \in G$. If $d(s) \neq c(t)$, then
$C_R(X)_s C_R(X)_t \subseteq R_s R_t = \{ 0 \}$.
Now suppose that $d(s) = c(t)$.
Since $C_R(X)$ is a subring of $R$
it follows that $$C_R(X)_s C_R(X)_t =
(C_R(X) \cap R_s)(C_R(X) \cap R_t) \subseteq
C_R(X)C_R(X) \subseteq C_R(X).$$
Since $R$ is graded it follows that
$$C_R(X)_s C_R(X)_t = (C_R(X) \cap R_s)(C_R(X) \cap R_t)
\subseteq R_s R_t \subseteq R_{st}.$$
Therefore $C_R(X)_s C_R(X)_t \subseteq
C_R(X) \cap R_{st} = C_R(X)_{st}$.

(b) This is a consequence of the following chain of equalities
$$C_R(A)_s = C_R(A) \cap R_s = C_{R_s}(A) = C_{R_s}\left(\bigoplus_{u \in G}
A_u\right) = \bigcap_{u \in G} C_{R_s}(A_u).$$

(c) It is clear that $C_R(R_0) \supseteq \bigoplus_{s \in G}
C_R(R_0)_s$. Now we show the reversed inclusion. Take $x \in
C_R(R_0)$, $e \in \ob(G)$ and $a_e \in R_e$. Then $\sum_{s \in G} x_s
a_e = \sum_{s \in G} a_e x_s$. By comparing terms of the same
degree, we can conclude that $x_s a_e = a_e x_s$ for all $s \in G$.
Since $e \in \ob(G)$ and $a_e \in A_e$ were arbitrarily chosen this
implies that $x_s \in C_R(R_0)_s$ for all $s \in G$. Now we show the
second part of (c). Take $e \in \ob(G)$. Suppose that $c(s)=d(s)$. If
$d(s) \neq e$, then $C_{R_s}(R_e) = R_s$.
Hence, by (b), we get that $C_R(R_0)_s = \bigcap_{e \in
\ob(G)}C_{R_s}(R_e) = C_{R_s}(R_{d(s)})$. Now suppose that $c(s) \neq
d(s)$. If $c(s) \neq e \neq d(s)$, then $C_{R_s}(R_e) = R_s$.
Therefore, by (b), we get that
$C_R(R_0)_s = \bigcap_{e \in \ob(G)}C_{R_s}(R_e) = C_{R_s}(R_{c(s)})
\bigcap C_{R_s}(R_{d(s)})$; $C_{R_s}(R_{c(s)})$ equals the set of
$r_s \in R_s$ such that $a r_s = r_s a$ for all $a \in R_{c(s)}$.
Since $d(s) \neq c(s)$, we get that $r_s a_e = 0$;
$C_{R_s}(R_{d(s)})$ is treated similarly.

(d) The claim follows immediately from (c). In fact, suppose that
$c(s) \neq d(s)$. Take $r_s \in R_s$ such that $R_{c(s)}r_s = \{ 0
\}$. Then $r_s = 1_R r_s = 1_{c(s)} r_s = 0$.
The last part follows from Proposition \ref{identityprop}.
\end{proof}

The essence of the following definition stems from 
the paper \cite{cohrow} by Cohen and Rowen.

\begin{defn}\label{defnonzeroidealproperty}
Let $R$ be a ring graded by a category $G$.
The $G$-grading of $R$ is said to be \emph{right nondegenerate} (resp.
\emph{left nondegenerate}) if to each isomorphism $s \in G$
and each nonzero $x \in R_s$, the
right (resp. left) $R_0$-ideal $x R_{s^{-1}}$ (resp. $R_{s^{-1}} x$) is nonzero.
\end{defn}

As the following example shows, a grading which is right nondegenerate
is not necessarily left nondegenerate, and vice versa.

\begin{exmp}
Let $u$ be a symbol and put $R=R_0 \oplus R_1$ where $R_0 = \frac{\Z_2[x]}{(x^2)}$ and $R_1 = R_0 u$. We define multiplication in the obvious way and put $u^2=0$ and $ux=0$. It is clear that $R$ is a $\Z_2$-graded ring and $x R_1 \neq \{0\}$ but $R_1 x = \{0\}$.
Thus, the grading is right nondegenerate, but not left nondegenerate.
\end{exmp}

\subsection*{Proof of Theorem \ref{maintheorem}}
We prove the contrapositive statement. Let $C$ denote the commutant
of $Z(R_0)$ in $R$ and suppose that $I$ is a twosided ideal of $R$
with the property that $I \cap C = \{ 0 \}$. We wish to show that $I
= \{ 0 \}$. Take $x \in I$. If $x \in C$, then by the assumption $x
= 0$. Therefore we now assume that $x  = \sum_{s \in G} x_s  \in I$,
$x_s \in R_s$, $s \in G$, and that $x$ is chosen so that $x \notin
C$ with the set $S := \{ s \in G \mid x_s \neq 0 \}$ of least
possible cardinality $N$. Seeking a contradiction, suppose that $N$
is positive. First note that there is $e \in \ob(G)$ with $1_e x \in
I \setminus C$. In fact, if $1_e x \in C$ for all $e \in \ob(G)$,
then $x = 1_R x = \sum_{e \in \ob(G)} 1_e x \in C$ which is a
contradiction. Note that, by the proof of Proposition \ref{identityprop}(a), the sum
$\sum_{e \in \ob(G)} 1_e$, and hence the sum $\sum_{e \in \ob(G)}
1_e x$, is finite. By minimality of $N$, we can assume that
$c(s)=e$, $s \in S$, for some fixed $e \in \ob(G)$. Take $t \in S$.
By the right (or left) nondegeneracy of the grading there is $y \in R_{t^{-1}}$ with $x_t
y \neq 0$ (or $y x_t \neq 0$). By minimality of $N$, we can therefore from now on assume
that $e \in S$ and $d(s) = c(s) = e$ for all $s \in S$. Take $d =
\sum_{f \in \ob(G)} d_f \in Z(R_0)$ where $d_f \in R_f$, $f \in
\ob(G)$ and note that $Z(R_0) = \bigoplus_{f \in \ob(G)} Z(R_f)$.
Then $I \ni dx - xd = \sum_{s \in S} \sum_{f \in \ob(G)} (d_f x_s -
x_s d_f)= \sum_{s \in S} d_e x_s - x_s d_e$. In the $R_e$ component
of this sum we have $d_e x_e -x_e d_e=0$ since $d_e \in Z(R_e)$.
Thus, the summand vanishes for $s = e$, and hence, by minimality of
$N$, we get that $dx-xd = 0$. Since $d \in Z(R_0)$ was arbitrarily
chosen, we get that $x \in C$ which is a contradiction. Therefore $N
= 0$ and hence $S = \emptyset$ which in turn implies that $x=0$.
Since $x \in I$ was arbitrarily chosen, we finally get that $I = \{
0 \}$.
{\hfill $\square$}

$ $

The next two results show that Theorem \ref{maintheorem}
is a simultaneous generalization of Theorem \ref{TheoremOne}
and Theorem \ref{TheoremTwo}.

\begin{prop}
The grading of any strongly groupoid graded ring is right (and left) nondegenerate.
In particular, the same conclusion
holds for the grading of any strongly group graded ring.
\end{prop}

\begin{proof}
Suppose that $R$ is a ring strongly graded by the groupoid $G$
and that we have chosen $s \in G$ and a nonzero $x \in R_s$.
By Proposition \ref{identityprop}(a) it follows that $1_R \in R_0$.
Hence $0 \neq x = x 1_R = x 1_{d(s)} \in x R_{s^{-1}}R_s$.
Therefore the right $R_0$-ideal $x R_{s^{-1}}$ is nonzero.
The proof of left nondegeneracy is done analogously.
The last part follows since a group is a groupoid.
\end{proof}

\begin{prop}
Suppose that $R = A \rtimes^{\sigma}_{\alpha} G$ is
crossed groupoid product such that $G$
has finitely many objects. If for each $s \in G$,
the element $\alpha(s,s^{-1})$ is not a zero divisor
in $A_{c(s)}$, then the standard grading of $R$ is right (and left) nondegenerate.
In particular, the same conclusion holds for the standard gradings of
crystalline graded rings.
\end{prop}

\begin{proof}
Take $s \in G$ and $x = au_s \in R_s$ for some nonzero $a \in A_{c(s)}$.
By Proposition \ref{identityprop}(b) it follows
that $1_R \in R_0$.
Hence, since $\alpha(s,s^{-1})$ is not a zero divisor in $A_{c(s)}$,
we get that $0 \neq a \alpha(s,s^{-1}) 1_R =
a \alpha(s,s^{-1}) u_{c(s)} = x u_{s^{-1}} \in x R_{s^{-1}}$.
Therefore, the right $R_0$-ideal $x R_{s^{-1}}$ is nonzero. Similarly one can show that the left $R_0$-ideal $R_{s^{-1}}x$ is nonzero.

Now suppose that $G$ is a group with identity element $e$.
It follows from \cite[Corollary 1.7]{nauoyst} and the discussion
preceding it that every crystalline
$G$-graded ring $R$ may be presented as $A \rtimes_{\alpha}^{\sigma} G$
where $A = R_e$, the maps $\sigma_s : A \rightarrow A$,
for $s \in G$, are ring automorphisms and for all
$s,t \in G$ the element $\alpha(s,t)$ is not a zero divisor in $R_e$
satisfying equations (\ref{idd})--(\ref{algebraa})
from Example \ref{example1} (note that $d(s)=c(t)=e$
since $G$ is a group).
\end{proof}

\begin{rem}
(a) Even if we restrict ourselves to the group graded situation,
there exist rings that are neither strongly graded nor
crystalline graded but still have a grading which is right nondegenerate.
In fact, let $K$ be any integral domain which is not a field;
take a nonzero element $\pi$ in $K$ which is not a unit.
Let $R = R_0 \oplus R_1$ be the decomposition as additive groups
given in Example \ref{example5} but with a new multiplication defined
by $e_{ij} e_{jk} = e_{ik}$, if $i=j$ or $j=k$, and
$e_{ij}e_{jk} = \pi e_{ik}$, otherwise.
Then $R$ is a ${\Bbb Z}_2$-graded ring
with a right nondegenerate grading. Indeed, suppose that
$0 \neq x = ae_{21} + be_{31} + ce_{12} + de_{13} \in R_1$
for some $a,b,c,d \in K$.
If $a \neq 0$ or $b \neq 0$, then $xR_1 \ni \pi a e_{22} + \pi b e_{32} \neq 0$
and hence $x R_1 \neq 0$.
The case when $c \neq 0$ or $d \neq 0$
is treated in a similar way.
By the equality $R_1 R_1 = \pi R_0$ it is clear that
$R$ is not strongly graded.
By an argument similar to the one used in
Example \ref{example5} it follows that $R$
is not crystalline graded.

(b) The groupoid graded ring in the end
of Example \ref{example5} can, in a similar way,
be used to construct a ring which has a right nondegenerate grading,
but which is neither
strongly graded nor a groupoid crossed product.
We leave the details in this construction to the reader.
\end{rem}

\begin{rem}
The ideal intersection property does
not hold for the commutant of the principal component
in arbitrary category graded rings.
In the following example we construct a group graded ring $R$, where $R_e$ is commutative, for which $C_R(R_e)$ does not have the ideal intersection property. This is made possible through relation (i) below, which ensures us that the grading of $R$ is not right nondegenerate nor left nondegenerate.
\end{rem}

\begin{exmp}
Let $\C[T]$ denote the polynomial ring in the indeterminate $T$. Choose some nonzero $q\in \C$ which is not a root of unity, i.e. $q^n \neq 1$ for all $n \in \Z \setminus \{0\}$. Consider the automorphism $\sigma$ of $\C[T]$ defined by
\begin{displaymath}
	\sigma : \C[T] \to \C[T], \quad p(T) \mapsto p(qT).
\end{displaymath}
Let $R$ denote the algebra generated by $\C[T]$ and two symbols $X$ and $Y$, subject to the following relations:
\begin{enumerate}
	\item[(i)] $XY=0$ and $YX=0$.
	\item[(ii)] $X p(T) = \sigma(p(T)) X$ for $p(T) \in \C[T]$.
	\item[(iii)] $Y p(T) = \sigma^{-1}(p(T)) Y$ for $p(T) \in \C[T]$.
\end{enumerate}
We can endow $R$ with a $\Z$-gradation by putting $R_0 = \C[T]$, $R_n = \C[T] X^n$ for $n>0$ and $R_n = \C[T] Y^{-n}$ for $n<0$. Furthermore, it is clear that $R_0 = \C[T]$ is maximal commutative in $R$ since $q$ is not a root of unity. It follows from (i) that the principal ideal $I=\langle X \rangle$ in $R$ has trivial intersection with $R_0$ and hence $R_0=C_R(R_0)$ does not have the ideal intersection property.
\end{exmp}

\subsection*{Proof of Theorem \ref{skewcatalg}}
We show the contrapositive statement. Suppose that $A$ is not
maximal commutative in $A \rtimes^{\sigma} G$.
Then, by Proposition \ref{gradedcommutant}(d),
there exists some $e \in \ob(G)$, some $s \in G \setminus \ob(G)$,
with $d(s)=c(s)=e$, and some nonzero $a \in A_e$, such that
$a u_s$ commutes with all of $A$.
Let $I$ be the nonzero ideal in $A \rtimes^{\sigma} G$ generated
by the element $a u_e - a u_s$
and define the homomorphism of abelian groups
$\varphi : A \rtimes^{\sigma} G \rightarrow A$
by the additive extension of the relation
$\varphi(x u_t) = x$, for $t \in G$ and $x \in A_{c(t)}$.
We claim that $I \subseteq \ker(\varphi)$.
If we assume that the claim holds, then,
since $\varphi |_A = {\rm id}_A$,
it follows that
$A \cap I = \varphi |_A (A \cap I)
\subseteq \varphi(I) = \{ 0 \}$.
Now we show the claim.
By the definition of $I$ it follows that
it is enough to show that $\varphi$
maps elements of the form
$x u_r (a u_e - a u_s) y u_t$ to zero,
where $x \in A_{c(r)}$, $y \in A_{c(t)}$
and $r,t \in G$ satisfy $d(r)=e=c(t)$.
However, since $au_s$ commutes with all
of $A$, we get that
$x u_r (a u_e - a u_s) y u_t =
x u_r (a y u_t - ay u_s u_t)=
x u_r (a y u_t - ay u_{st})=
x \sigma_r(ay) u_{rt} - x \sigma_r(ay) u_{rst}$
which, obviously, is mapped to zero by $\varphi$.
{\hfill $\square$}

\section{Applications: Injectivity of ring morphisms}\label{Application_Injectivity}

The ideal intersection property is characterized by the following important proposition,
whose proof is easy and therefore omitted.

\begin{prop}\label{IdealChar}
Let $R'$ be a subring of a ring $R$ such that $R'$ has the ideal intersection property, then
for any ring $S$ and any ring morphism $\varphi : R \to S$, the following assertions are equivalent:
\begin{enumerate}
	\item[{\rm (i)}] $\varphi : R \to S$ is injective.
	\item[{\rm (ii)}] $\varphi\lvert_{R'} : R' \to S$ is injective. (The restriction of $\varphi$ to $R'$.)
\end{enumerate}
Conversely, if $R'$ is a subring of a ring $R$ such that for any
ring $S$ and any ring morphism $\varphi : R \to S$, properties {\rm (i)} and {\rm (ii)} above
are equivalent, then $R'$ has the ideal intersection property.
\end{prop}

We now obtain the following results as corollaries to Theorem \ref{maintheorem} respectively Theorem \ref{skewgroupoidtheorem}.

\begin{cor}
If the grading of a groupoid graded ring $R=\bigoplus_{g\in G} R_g$ is right (or left) nondegenerate, then
for any ring $S$ and any ring morphism $\varphi : R \to S$, the following assertions are equivalent (where $R_0 = \bigoplus_{e\in \ob(G)} R_e$):
\begin{enumerate}
	\item[{\rm (i)}] $\varphi : R \to S$ is injective.
	\item[{\rm (ii)}] $\varphi\lvert_{C_R(Z(R_0))} : C_R(Z(R_0)) \to S$ is injective.
\end{enumerate}
\end{cor}

\begin{cor}
Suppose that $R$ is a skew groupoid
algebra with a commutative principal
component $A$ and that the grading of $R$ is defined by a
groupoid with finitely many objects. Then the following assertions are equivalent:
\begin{enumerate}
	\item[{\rm (i)}] $A$ is maximal commutative in $R$.
	\item[{\rm (ii)}] For any ring $S$ and any ring morphism $\varphi : R \to S$, the following assertions are equivalent:
\begin{enumerate}
	\item[{\rm (a)}] $\varphi : R \to S$ is injective.
	\item[{\rm (b)}] $\varphi\lvert_{A} : A \to S$ is injective.
\end{enumerate}
\end{enumerate}

\end{cor}

\end{document}